\documentclass{article}
\usepackage{amsmath,amsfonts,amssymb,mathrsfs,amsthm,tikz-cd,url,hyperref}  
\newtheorem{thm}{Theorem}[section]

\theoremstyle{plain}
\newtheorem{prop}[thm]{Proposition}

\newtheorem{lem}[thm]{Lemma}
\newtheorem{cor}[thm]{Corollary}

\theoremstyle{definition}
\newtheorem{defn}[thm]{Definition}

\theoremstyle{remark}
\newtheorem{rem}[thm]{Remark}

\title{The 6-step Solvable Mono-anabelian Reconstruction of Abelian Number Fields}
\author{Yu Mao}
\date{}
\begin{document}
\maketitle
\begin{abstract}
In this paper, we develop a new method to reconstruct an abelian number field $K$ from the maximal $6$-step solvable quotient of $G_K$ group-theoretically. The new aspect of this paper is that the results in this paper are independent from the bi-anabelian results proved in \cite{ST1}.
\end{abstract}
\newpage
\section*{Acknowledgement}
The author would like to thank to Ivan Fesenko and Mohamed Sa\"idi for their valuable comments and discussions.
\section*{Conventions}
Throughout this paper, we shall use the following notions.
\begin{itemize}
\item We write $\Sigma$ for the set of all prime numbers.
\item Let $K$ be a number field, we write $\mathscr{P}_K^{\text{fin}}$ for the set of all finite places of $K$. More generally, if $K/\mathbb{Q}$ is an algebraic extension, then we also write $\mathscr{P}_K^{\text{fin}}$ for the set of all non-archimedean primes of $K$. In particular, if $K/\mathbb{Q}$ is infinite, we have the natural identification $\mathscr{P}_K^{\text{fin}} = \varinjlim_{F} \mathscr{P}_F^{\text{fin}}$ where $F$ ranges over all finite extensions of $\mathbb{Q}$ contained in $K$.
\item Let $G$ be a profinite group. We write $G^{\text{sol}}$ for the maximal pro-solvable quotient of $G$ and $G^{\text{ab}}$ for the maximal abelian quotient of $G$.
\item Let $G$ be a profinite group. We write $G^{[i]}$ for the (closure of the) $i$-th derived subgroup of $G$, i.e. $G^{[0]} := G$, $G^{[i]} := \overline{[G^{[i-1]},G^{[i-1]}]}$ for $i \geq 1$. We write $G^i$ for the maximal $i$-step solvable quotient of $G$. In particular, $G^1 = G^{\text{ab}}$. Moreover, we have the following canonical identification $G^{\text{sol}} = \varprojlim_i G^i$.
\item Let $K$ be a field, fix an algebraic closure $\overline{K}$ of $K$. We write $K^{\text{sep}}$ for the separable closure of $K$ contained in $\overline{K}$. We write $G_K := \text{Gal}(K^{\text{sep}}/K)$ for the absolute Galois group of $K$.
\item Let $K$ be a field, we write $K^{\text{sol}} := (K^{\text{sep}})^{\text{ker}(G_K \twoheadrightarrow G_K^{\text{sol}})}$ for the maximal pro-solvable extension of $K$ contained in $\overline{K}$. Moreover, for each $i \geq 1$, we write $K_i := (K^{\text{sep}})^{G^{[i]}}$ for the maximal $i$-step solvable extension of $K$. In particular, $K_1 = K^{\text{ab}}$ is the maximal abelian extension of $K$.
\item All homomorphisms between profinite groups are assumed to be continuous.
\item Let $A$ be an abelian group, we write $A^{\wedge} := \varprojlim A/nA$, viewed additively.
\item Let $A$ be an abelian group, we write $A_{\text{tor}}$ for the torsion subgroup of $A$, and $A^{/\text{tor}} := A/A_{\text{tor}}$.
\end{itemize}
\section{Introduction}
First, let us recall the $m$-step solvable version of the Neukirch-Uchida Theorem:
\begin{thm}[Sa\"idi-Tamagawa, Theorem 2 in \cite{ST1}]
Let $K,L$ be number fields. Let $m \geq 0$ be an integer. Let $\tau_{m+3}: G_K^{m+3} \xrightarrow{\sim} G_L^{m+3}$ be an isomorphism. Then the followings hold: \par 
(i) There exists a field isomorphism $\sigma_{m} : K_m \xrightarrow{\sim} L_m$ such that $\tau_m(g) = \sigma_mg\sigma_m^{-1}$ for every $g \in G_K^m$, where $\tau_m$ is the isomorphism between $G_K^m$ and $G_L^m$ induced by $\tau_{m+3}$. \par 
(ii) Assume that $m \geq 2$ (resp. $m=1$). Then the isomorphism $\sigma_m: K_m \xrightarrow{\sim} L_m$ is unique (resp. $\sigma: K \xrightarrow{\sim}L$ induced by $\sigma_1: K_1 \xrightarrow{\sim} L_1$).
\end{thm}
But neither the statement of Theorem 1.1 nor the proof of Theorem 1.1 provides an explicit reconstruction of the number field $K$ from $G_K^{m}$ for suitable $m$. \par 
On the other hand, in \cite{MS}, there is a mono-anabelian version of Theorem 1.1, i.e. a group theoretic reconstruction of $K_m/K$ starting from $G_K^{m+9}$. However, the result proved in \cite{MS} uses the statement of Theorem 1.1 in an essential way. In particular, Theorem 1 in \cite{MS} does not provide a weaker alternative proof to Theorem 1.1. \par 
On the other hand, Theorem 2 in \cite{MS} does not depend on Theorem 1.1, in particular, Theorem 2 in \cite{MS} does provide an alternative proof to a weaker version of Theorem 1.1 for subfields of imaginary quadratic fields. \par 
In this paper, we develop a new method to recover all abelian number fields from their maximal $6$-step solvable Galois groups. Our method in the present paper is independent from the statement of Theorem 1.1, and different from Theorem 2 in \cite{MS}. In particular, we prove:
\begin{thm}
Let $K$ be an abelian number field (i.e. a finite abelian extension of $\mathbb{Q}$). Then there is a group-theoretic reconstruction of $K$ from $G_K^6$.
\end{thm}
\section{Some Local Reconstructions}
In this section, we review the local theory proven by Sa\"idi and Tamagawa in \cite{ST1} and some local reconstructions related to decomposition groups, basically everything in this section are proven in \cite{ST1} and \cite{MS}.
\begin{defn}[c.f. Definition 3.2 in \cite{Ho1}]
A profinite group $G$ is said to be of NF-type if it admits the following data: \par 
(i) A number field $K$ together with its separable closure $K^{\text{sep}}$. \par 
(ii) An isomorphism $\alpha: G \xrightarrow{\sim} G_K = \text{Gal}(K^{\text{sep}}/K)$ of profinite groups. \par 
In this case, we also say that $G$ is a profinite group of NF-type associated to $K$.
\end{defn}
In particular, a profinite group of NF-type should be thought of as an abstract profinite group only, but not a profinite group naturally arises from field theory or scheme theory.
\begin{defn} [c.f. Definition 3.2 in \cite{Ho1}, Definition 2.1 in \cite{MS}]
A profinite group of \text{GSC}-type (resp. $\text{GSC}^m$-type for some positive integer $m$) is a profinite group $G$ isomorphic to the maximal prosolvable (resp. $m$-step solvable quotient) quotient of some profinite group of NF-type.
\end{defn}
We shall recall one of the main results in \cite{ST1}, which will play an important role in our reconstruction:
\begin{thm}[c.f. Theorem 1.25 in \cite{ST1}]
Let $K$ be a number field and $m\geq1$ be an integer. Then the set of decomposition subgroups at non-archimedean primes of $G_K^m$ can be group-theoretically from $G_K^{m+2}$.
\end{thm}
Since Theorem 2.3 in mono-anabelian, we may adapt the statement to Definition 2.2.
\begin{defn}[c.f. Definition 2.3 in \cite{MS}]
Let $m\geq 1$ be an integer. Let $G$ be a profinite group of $\text{GSC}^{m+2}$-type. We write $\widetilde{\text{Dec}}(G^m)$ for the $G^m$-set reconstructed from $G^{m+2}$ by using Theorem 2.3. Moreover, we write $\text{Dec}(G^1) := \widetilde{\text{Dec}}(G^m)/G^m$ where the $G^m$-action is given by conjugation.
\end{defn}
\begin{prop}
Let $G$ be a profinite group of $\text{GSC}^{m+2}$-type for some positive integer $m$, then the following diagram commutes:
$$
\begin{tikzcd}
\widetilde{\text{Dec}}(G^m) \arrow[r,"\sim"] \arrow[d,twoheadrightarrow,"/G^m"] & \mathscr{P}_{K_m}^{\text{fin}} \arrow[d,twoheadrightarrow,"/G_K^m"] \\
\text{Dec}(G^1) \arrow[r,"\sim"] & \mathscr{P}_K^{\text{fin}}
\end{tikzcd}
$$
where the horizontal arrows are determined by the isomorphism $\alpha_m: G \xrightarrow{\sim} G_K^m$ together with Proposition 1.5 in \cite{ST1}.
\end{prop}
\begin{proof}
See Corollary 1.6 in \cite{ST1}.
\end{proof}
\begin{lem}
Let $m,n$ be positive integers. Let $G$ be a profinite group of $\text{GSC}^{m+n}$-type, and let $H \subset G^m$ be an open subgroup. We write $\widetilde{H}$ for the inverse image of $H$ in $G$. Then $\widetilde{H}^n$ is a profinite group of $\text{GSC}^n$-type.
\end{lem}
\begin{proof}
This is an immediate consequence of infinite Galois theory and the definition of inverse image.
\end{proof}
\begin{defn}
If $G$ is a profinite group of $\text{GSC}^{m+n}$-type for some positive integers $m,n$, and $H \subset G^m$ is an open subgroup. Then we shall write $\widetilde{\text{Dec}}(H) := \widetilde{\text{Dec}}(\widetilde{H}^n)$, and $\text{Dec}(H) := \widetilde{\text{Dec}}(\widetilde{H}^n)/\widetilde{H}^n$.
\end{defn}
\begin{prop}[c.f. Theorem 2.6 in \cite{MS}]
Let $G$ be a profinite group of $\text{GSC}^6$-type. Let $D \in \widetilde{\text{Dec}}(G^4)$. It follows from Proposition 1.1 (i) in \cite{ST1} that $D \xrightarrow{\sim} D_{\mathfrak{p}_4}$ for some uniquely determined $\mathfrak{p}_4 \in \mathscr{P}_{K_4}^{\text{fin}}$ and $\mathfrak{ p}$ is the image of $\mathfrak{p}_4$ in $\mathscr{P}_K^{\text{fin}}$. Then we can group-theoretically reconstruct the following objects starting from $D$: \par 
(i) The residue characteristic $p_{\mathfrak{p}}$ of $K_{\mathfrak{p}}$, the inertia degree of $f_\mathfrak{p}$ of $K_{\mathfrak{p}}$, the degree $d_{\mathfrak{p}} := [K_{\mathfrak{p}}:\mathbb Q_{p(D)}]$ and the absolute ramification index $e_{\mathfrak{p}}$ of $K_{\mathfrak{p}}$. \par 
(ii) The inertia subgroup $I_{\mathfrak{p}_2}$ of $D_{\mathfrak{p}_2}$ and the wild inertia subgroup $W_{\mathfrak{p}_2}$ of $I_{\mathfrak{p}_2}$. \par 
(iii) The Frobenius element $\text{Frob}_{\mathfrak{p}}$ contained in $G_{K_{\mathfrak{p}}}^{\text{unr}}$. \par 
(iv) The multiplicative group $K_{\mathfrak{p}}^{\times}$. \par 
(v) The group of roots of unity $\mu(K_{\mathfrak{p}}^{\text{ab}})$ contained in $K_{\mathfrak{p}}^{\text{ab}}$ and the cyclotome $\Lambda(K_{\mathfrak{p}}^{\text{ab}})$ associated to $K_{\mathfrak{p}}^{\text{ab}}$ as a $D$-module.
\end{prop}
\begin{proof}
C.f. Theorem 2.6 in \cite{MS}. We also use the same notations as in Theorem 2.6 in \cite{MS}.
\end{proof}
\begin{cor}
Let $G$ be a profinite group of $\text{GSC}^6$-type associated to a number field $K$, and $\mathfrak{p} \in \mathscr{P}_K^{\text{fin}}$. Then whether or not $K_\mathfrak{p}$ is abelian over $\mathbb{Q}_{p_{\mathfrak{p}}}$ can be group-theoretically determined by $D \in \widetilde{\text{Dec}}(G^4)$, for $D$ corresponds to $\mathfrak{p}$.
\end{cor}
\begin{proof}
This follows from Proposition 2.8 (v) in \cite{MS} together with Lemma 4.6 and Proposition 4.9 in \cite{Ho4}.
\end{proof}
\begin{rem}
In fact, by local class field theory, we can even reconstruct the Galois group $\text{Gal}(K_{\mathfrak{p}}/\mathbb{Q}_{p_D})$ group-theoretically provided that $K_{\mathfrak{p}}$ is abelian. The Galois group is isomorphic to $\mathbb{Q}_{p_D}^{\times}/\text{Nm}(k^{\times}(D^1))$ (c.f. Definition 4.7 in \cite{Ho4} for the definition of $\text{Nm}(k^{\times}(D^1))$), which is group-theoretically determined by $D$. We shall write $\underline{D}$ for this finite group.
\end{rem}
\section{Reconstruction of Abelian Number Fields}
Let $K$ be a finite abelian extension of $\mathbb Q$ for the rest of this section, and $G$ be a profinite group of $\text{GSC}^6$-type associated to $K$. We aim to prove Theorem 1.2 in this section.
\begin{prop}
Let $D \in \widetilde{\text{Dec}}(G^4)$. There exists a group-theoretic reconstruction of a topological field $k(D)$ starting from $D$ such that
$$
k(D) \xrightarrow{\sim} K_{\mathfrak{p}}.
$$
\end{prop}
\begin{proof}
It follows from Proposition 2.8 (v) and Lemma 4.6 (i), that we can group-theoretically reconstruct the topological field $\mathbb Q_{p_D}$. \par 
Let $n \geq 1$ be an integer, consider $\mu(D^1)[n]$ as a sub-$D^1$-module. We shall write
$$
\mathcal{A}_n(D) := \mathbb Q_{p_D}[\mu(D^1)[n]]
$$
for the group-ring over $\mathbb Q_{p_D}$. It follows from Maschke's theorem and Wedderburn-Artin theorem that
$$
\mathcal{A}_n(D) \xrightarrow{\sim} \prod_d \mathbb Q_{p_D}(\zeta_d)
$$
\end{proof}
for some integer $d$, and each component of $\mathcal{A}_n(D)$ is a finite field extension of $\mathbb Q_{p_D}$ obtained by adjoining the primitive $d$-th roots of unity (and notice that if $d|(p_D-1)$, the extension is trivial). In particular, if $\mathcal{A}_n(D)$ always admits a quotient $\mathbb Q_{p_D}$-algebra isomorphic to $\mathbb Q_{p_D}(\zeta_n)$ (this can be characterised as the $\mathbb Q_{p_D}$-quotient algebra of $\mathcal{A}_n(D)$ which is a finite field extension of $\mathbb Q_{p_D}$ of maximal degree), we shall denote by $F_n(D)$ for this quotient. Thus, we can conclude that there is a group-theoretic reconstruction of the fields $F_n(D)$ for each $n \geq 1$ as field extensions over $\mathbb Q_{p_D}$. \par 
Now by local class field theory for $\mathbb Q_{p_D}$, every finite abelian extension of $\mathbb Q_{p(D)}$ is contained in $F_n(D)$ for some $n$. Hence the finite group $\underline{D}$ is a finite quotient of $\text{Gal}(F_{\infty}(D)/\mathbb Q_{p_D})$ where $F_{\infty}(D) := \bigcup_{n \geq 1} F_n(D)$. Write
$$
k(D) := F_{\infty}(D)^{\text{ker}(\text{Gal}(F_{\infty}(D)/\mathbb Q_{p_D}) \twoheadrightarrow \underline{D})}
$$
where $\underline{D}$ is defined in Remark 2.10. The desired isomorphism $k(D) \xrightarrow{\sim} K_{\mathfrak{p}}$ follows from Proposition 2.5 and Proposition 2.8 together with the various constructions involved.
\begin{defn}
Let $v \in \text{Dec}(G^1)$ be the image of $D \in \widetilde{\text{Dec}}(G^4)$ (by Proposition 2.5). We define: \par 
\begin{itemize}
\item $p(v) = p(D)$. 
\item $k(v) := k(D)$ (c.f. Proposition 3.1, and Definition 3.1 in \cite{MS}).
\end{itemize}
\end{defn}

\begin{prop}
There is a group-theoretic reconstruction of a field $F_{\infty}(G)$ from $G$ such that there is an isomorphism of fields
$$
F_{\infty}(G) \xrightarrow{\sim} \mathbb{Q}^{\text{ab}}
$$
which is equivariant w.r.t $\alpha: G \xrightarrow{\sim} G_K^6$.
\end{prop}
\begin{proof}
Notice that by Definition 3.5 in \cite{MS} together with Proposition 2.8, there is a group-theoretic reconstruction of $\mu(G^{1}) \cong \mu_{\infty}$ from $G$. It follows from Definition 2.2 in \cite{Ho2} (together with the various definitions in this paper involved), that there is a group-theoretic reconstruction of the prime subfield $\mathbb{Q}$ from $G$. \par 
By Proposition 3.1, together with the notion of Definition 3.2, we write
$$
F_{\infty}(v) := F_{\infty}(D)
$$
for some $v \in \text{Dec}(G^{1})$, where $v$ is the image of $D \in \widetilde{\text{Dec}}(G^4)$ in $\text{Dec}(G^1)$. Then one verifies immediately that we have the containment
$$
\mu(G^{1}) \subset \prod_{v \in \text{Dec}(G^1)} F_{\infty}(v)
$$
hence we write
$$
F_{\infty}(G) \subset \prod_{v \in \text{Dec}(G^{1})} F_{\infty}(v)
$$
for the subring of $\prod_{v \in \text{Dec}(G^{1})} F_{\infty}(v)$ generated by $\mu(G^{1})$ over $\mathbb{Q}$. Then this proposition follows immediately from the constructions above together with Theorem 3.7 in \cite{MS}.
\end{proof}
In this case, we shall write $\Gamma:= \text{Aut}_{\text{field}}(F_{\infty}(G))$, which is isomorphic to $G_{\mathbb{Q}}^{\text{ab}} \xrightarrow{\sim}\widehat{\mathbb{Z}}^{\times}$. Moreover, by considering the cyclotomic character, we obtain a natural map $G^{1} \to \Gamma$. \par

\begin{thm}
There is a group-theoretic reconstruction to a field $F(G)$ isomorphic to $K$.    
\end{thm}
\begin{proof}
It follows from Proposition 3.1, that we can reconstruct the fields $k(v)$ for any $v \in \text{Dec}(G^1)$. Since we have field structure, we write
$$
n_G := \prod_{v \in \text{Dec}(G^1)} \mathfrak{f}(v)
$$
where $\mathfrak{f}(v)$ is defined to be $s(v) + 1$ where $s(v)$ is the maximal integer such that the $s(v)$-th ramification group of $\text{Gal}(k(v)/\mathbb{Q}_{p(v)})$ in lower numbering is non-trivial. Since we have reconstructed the field structures on $k(v)$ and $\mathbb{Q}_{p(v)}$ group-theoretically (c.f. Corollary 2.9 and Proposition 3.1), the integer $s(v)$ is also regarded as group-theoretically determined by $D$. Hence the integer $n(G)$ can be group-theoretically determined by $G$. \par 
Then we write 
$$
F_n(G) := \mathbb{Q}(\mu(G^{1})[n])
$$
where $\mu(G^{1})$ (c.f. Definition 3.5 in \cite{MS} and Proposition 3.6 in \cite{Ho1}), which is isomorphic to $\mu_{\infty} \cong \mathbb{Q}/\mathbb{Z}(1)$. Recall that the cyclotomic character $\chi: G^{\text{ab}} \to \Gamma$ can be group-theoretically reconstructed from $G^3$ by Theorem 1.26 in \cite{ST1}, hence we can also reconstruct the mod-$n_G$ cyclotomic character, hence we have the map
$$
\chi_{\text{mod}-n_G}: \Gamma \twoheadrightarrow (\mathbb{Z}/n_G\mathbb{Z})^{\times},
$$
also we have the natural identification
$$
F_{n}(G) = F_{\infty}(G)^{\text{ker}(\chi_{\text{mod}-n_G})}.
$$
Next, we define
$$
H_{\text{cycl}} := G^{1} \times_{\Gamma} \text{ker}(\chi_{\text{mod}-n_G}).
$$
One verifies that $H_{\text{cycl}}$ is an open subgroup of $G^{1}$, we shall define
$$
F(G) := F_{n}(G)^{G^{1}/H_{\text{cycl}}}.
$$
One verifies easily that $F(G) \xrightarrow{\sim} K$. More precisely, this follows from that $n := n_G$ coincide with the product of the local conductor exponents of $K$, hence $\mathbb{Q}(\zeta_n)$ is the minimal cyclotomic field containing $K$ by Proposition 6.5 and Proposition 6.7 in Chapter VI of \cite{ANT}. One then verifies immediately that $\mathbb{Q}(\zeta_n)$ and $F_n(G)$ are isomorphic, hence the result follows.
\end{proof}

\bibliography{ref.bib}
\bibliographystyle{alphaurl}
\end{document}